\documentclass[12pt]{article}
\usepackage{amsmath, amsthm, amscd, amsfonts, amssymb, graphicx, color}
\usepackage{tikz}

\newtheorem{defn}{Definition}[section]
\newtheorem{lem}{Lemma}[section]

\newtheorem{thm}[lem]{Theorem}

\newcommand{\be}{\begin{equation}}
	\newcommand{\ee}{\end{equation}}
\begin{document}
	\setlength{\unitlength}{1mm}
	\baselineskip 8mm
	%=============================================

\thispagestyle{empty}
\title{
On Clique Incidence Matrices and Derivatives of Clique Polynomials}
\author
{Hossein Teimoori Faal
	%\thanks{}
	\vspace{.25in}\\
	Department of Mathematics and Computer Science,\\ 
	Allameh Tabataba'i University, Tehran, Iran
}

	\maketitle
	
\begin{abstract}

The ordinary generating function of the number of complete subgraphs (cliques) of $G$, denoted by $C(G,x)$, is called the 
The clique polynomial of the graph 
$G$.
In this paper, we first introduce some \emph{clique} incidence matrices associated by a simple graph $G$ as a generalization of the
classical 
vertex-edge incidence matrix of $G$. Then, using 
these clique incidence matrices, we obtain two clique-counting 
identities that can be used for deriving two combinatorial formulas for the first and the second derivatives of clique polynomials. Finally, we conclude 
the paper with several open questions and conjectures
about possible extensions of our main results for higher derivatives of 
clique polynomials. 	
	
\end{abstract}

\section{Introduction}

The 
\emph{reconstruction} problems in discrete setting can be considered as a theoretical foundation for many 
inverse problems in the area of computer science and 
information technology. In graph-theoretical setting, 
the well-know Kelly-Ulam \cite{Harary74}
\emph{reconstruction conjecture} simply states that any graph with at least three vertices can be reconstructed 
from its \emph{vertex-deck}. We recall that a vertex deck 
of a given graph $G$ is the collection of (not necessarily distinct) \emph{vertex-deleted} subgraphs; that is  
$\{G-v\}_{v \in V(G)}$. We will denote this collection by $deck_{v}(G)$. It seems that finding the general solution for this longstanding open problem is 
a quite challenging problem. 
\\
In another direction, one can try to reconstruct 
the key \emph{invariants} of a given graph rather than 
reconstructing the graph itself. 
It seems that the collection of complete subgraphs (that
we call them cliques), plays an essential role in 
reconstructing key invariants of graphs. 
In this respect, 
one of our main goal in this paper is to obtain some 
interesting \emph{clique-counting} identities. 
A common approach to tackle these kind of problems is 
to use the two ideas of \emph{incidence matrices}  
and the \emph{double-counting} method. 
In the classical literature of graph theory, the classical 
vertex-edge incidence matrix has been already introduced. A generalization of this idea has been also 
introduced in the context of \emph{graph homology} \cite{Giblin1977}
under the name of \emph{boundary maps} of \emph{simplicial complexes} as higher-dimensional generalizations 
of graphs. Indeed, one can construct edge-triangle incidence matrix and their higher-order cliques generalizations simply by defining a $(0,1)$-matrix 
where non-zero entries shows that a lower-dimensional clique 
(here we call it the \emph{subclique}) contained in 
is a subgraph of a higher-dimensional one (here we call it the \emph{superclique}).      
\\
In this paper, based on the motivation originating form  the reconstruction problem, we construct other  kinds of 
incidence matrices which we call them the \emph{clique incidence} matirces.   	 
Using this idea, we obtain two important clique-counting identities that relates 
the number of $k$-cliques ($k=1,2$) in the graph $G$
itself to the number of $k$-cliques in 
the \emph{clique-deck} of $G$. 
Then, we establish two combinatorial formulas for the first and the second derivatives of clique polynomials. 
A formula very similar to that of our first derivative has been already mentioned in the literature \cite{Li-Gutman}, but our formula for the second derivative of 
clique polynomials, to the best of our knowledge, is 
a new combinatorial formula. Indeed, there is no formula in the literature for the second or higher derivatives of many graph polynomials.   
For this reason, we also include a discussion about 
possible future directions in this fascinating area of clique-counting graph polynomials.

\section{
	The Generalized Incidence Matrices}

All graphs here are finite, simple and undirected. 
For the terminologies not defined here, one can consult the reference \cite{gtwa}\\
For a given graph $G=(V,E)$ and a vertex $v\in V(G)$,
it's \emph{vertex-deleted} subgraph denoted by 
$G-v$ is defined as an induced subgraph obtained from 
$G$ by deleting the vertex $v$. The collection of all 
vertex-deleted subgraphs of $G$ is called a 
\emph{vertex-deck} of 
$G$ and denoted by $deck_{v}(G)$. 
One can similarly define the \emph{edge-deleted} subgraph 
of $G$ which we denote it by $G-e$. The collection of all
edge-deleted subgraphs of $G$ is called a 
\emph{edge-deck} 
of $G$ and is denoted by $deck_{e}(G)$.  
The \emph{open neighborhood} 
of the vertex $v$ in $G$ is the set of vertices 
\emph{adjacent} to $v$ and is denoted by $N_{G}(v)$. 
A $k$-\emph{clique} of a graph $G$ 
is denoted by $Q_{k}$ is defined as a \emph{complete} subgraph $G$ on $k$ vertices. The set $\Delta_{k}(G)$
denotes the set of all $k$-cliques of $G$. 
We will also denote the number of $k$-cliques 
of $G$ by $c_{k}(G)$. 
In general, for any clique $Q_{k} \in \Delta_{k}(G)$,
we define the \emph{clique-deck} of $G$ as the collection
of all \emph{clique-deleted} subgraphs of $G$. 
\\
A generalization of the concept of the \emph{degree} of 
a vertex $v \in V(G)$ is called the \emph{clique-value} 
which can be defined, as follows.

\begin{defn}

Let $G=(V,E)$ be a graph and let $Q_{k}$ be a 
$k$-clique of $G$. Then, we define the \emph{clique-value} 
denoted by $val_{G}(Q_{k})$, as follows

\begin{equation}
val_{G}(Q_{k}) = 
\bigcap_{v \in V(Q_{k})} N_{G}(v). 
\end{equation}

\end{defn}

As a generalization of standard vertex-edge incidence 
matrix, we define our 
first clique incidence matrix that 
we call it 
the \emph{subclique-superclique} incidence matrix 
$I_{c,k}(G)$, as follows. 
From now on, we will assume 
$
\Delta_{k}(G)= 
\{
Q_{k,1}, Q_{k,2}, \ldots, Q_{k,r}
\}
$
in which $r=c_{k}(G)$.

\begin{defn}
	
For a given graph 
$G=(V,E)$, 
we define the \emph{subclique-superclique} incidence matrix of order $k$ denoted by  
$I_{c,k}(G)$, as follows 

\begin{equation}
\big( I_{c,k}(G)
\big)_{Q_{k,i},Q_{k+1,j}} =
\left\{ 
\begin{array}{ll}
1 & \textnormal{if $
	Q_{k,i}
	$~is~a~subgraph~of~
	$Q_{k+1,j}$}, \\
0 & \textnormal{otherwise}
\end{array} 
\right. 
\end{equation}

\end{defn}

It is clear that the subclique-superclique matrix of 
order $1$ is exactly the classic vertex-edge incidence 
matrix $I(G)$ of a graph $G$. 

\begin{lem}
	
	Let $G=(V,E)$ be a graph. Then, we have 
	
	\begin{equation}
	\sum_{Q_{k} \in \Delta_{k}(G)} val_{G}(Q_{k})=
	(k+1)c_{k+1}(G). 
	\end{equation}
	
\end{lem}

\begin{proof}
	
We first note that the number of $k$-cliques of 
any $(k+1)$-clique $Q_{k+1,j}$ is equal to $k+1$. We also
note that the number of times that any given 
$k$-clique $Q_{k}$ appears as a subset of another 
$(k+1)$-clique is equal to $val_{G}(Q_{k})$. 
Now, considering the definition of \emph{subclique-superclique} incidence matrix 
$I_{c,k}(G)$ and the double-counting technique, we get the desired result.

\end{proof}

We also need to recall the definition and basic 
properties of \emph{clique polynomials} \cite{HajiMehrabadi(1998)}. 
\\
The \emph{clique polynomial} of a given graph 
$G=(V,E)$ is defined as the ordinary generating 
function of the number $k$-cliques of $G$. More
precisely, we have 

\begin{equation}
C(G,x) = 1 + 
\sum_{k=1}^{\omega(G)} c_{k}(G) x^{k}
,
\end{equation}
where $\omega(G)$ is the size of the \emph{largest}
clique in $G$. 
\\
One can easily prove that the polynomial $C(G,x)$
satisfies the following \emph{vertex-recurrence} relation. 

\begin{equation} \label{vert-rec1}
C(G,x) = C(G-v,x) + x C(G[N_{G}(v)],x),
\hspace{0.4cm}(v \in V(G)). 
\end{equation}

Similarly, we can prove the following \emph{edge-recurrence} 
relation for clique polynomials. 

\begin{equation} \label{edge-rec2}
C(G,x) = C(G-e,x) + x^{2} C(G[N_{G}(e)],x),
\hspace{0.4cm}(e=\{u,v\} \in E(G)), 
\end{equation}
where by $N_{G}(e)$, we mean the subgraph of
$G$ induced by the common neighbors of the end vertices 
of $e$. In other words, we have 
$
N_{G}(e) = N_{G}(u) \cap N_{G}(v)
$
.

\section{Main Results} 

In this section, we first obtain two interesting 
\emph{clique-counting} identities related to the 
\emph{vertex-deck} and the \emph{edge-deck} of $G$. 
\\
We start with the following identity closely related 
to the \emph{vertex-deck} of $G$.

\begin{lem}\label{vertkey1}
	Let $G=(V,E)$ be a graph on $n$ vertices. Then, we have 
	
	\begin{equation}
	(n-k)c_{k}(G) =
	\sum_{v \in V(G)}c_{k}(G-v) \hspace{0.4cm}
	(k\geq 1).  
	\end{equation}
	
\end{lem}

\begin{proof}
	Put $r=c_{k}(G)$.  	
	Let $\Delta_{k}(G)=
	\{Q_{k,1}, Q_{k,2}, \ldots, Q_{k,r}\}
	$
	be the set of all $k$-cliques of $G$. 
	Now, we consider our second clique incidence matrix 
	$I_{H,k}(G)$ which is defined, as follows

\begin{equation}
\big(I_{k,r}(G)
\big)_{Q_{k,i},H-v_{j}} = 
\left\{ 
\begin{array}{ll}
		1 & \textnormal{if 
			$
			Q_{k,i}
			$~is~a~subgraph~of~
			$H-v_{j}$
		}, \\
		0 & otherwise. 
\end{array} 
\right. 
\end{equation}

Now, it is obvious that the sum of rows is 
equal to $(n-k)c_{k}(G)$ and the sum of columns equals
to 
$
\sum_{v \in V(G)}c_{k}(G-v)
$
. 
Hence, the proof is complete based on  
the \emph{double-counting} method.

\end{proof}

In a similar way, one can prove the following \emph{clique-counting} identity in close connection with 
\emph{edge-deck} of a graph. 

\begin{lem}\label{edgekey2}
	Let $G=(V,E)$ be a graph on $n$ vertices
	and $m$ edges. Then, we have 
	
	\begin{equation}
	\Big(m-{k \choose 2}\Big)c_{k}(G) =
	\sum_{e \in E(G)}c_{k}(G-e), 
	\hspace{0.4cm}(k\geq 2). 
	\end{equation}
	
\end{lem}

The following \emph{graph-theoretical} interpretation of the \emph{first derivative} of clique polynomials is similar to that of 
\cite{Li-Gutman}. Here, we 
also include the proof which is based on Lemma \ref{vertkey1}.

\begin{thm}
	Let $G=(V,E)$ be a graph. 	
	Then, we have 
	
	\begin{equation}
	\frac{d}{dx}C(G,x) = \sum_{v \in V(G)}C(G[N(v)], x ).  
	\end{equation}
	
\end{thm} 

\begin{proof}
	
	By multiplying both sides of formula (\ref{vert-rec1}) by 
	$x^{i}$ and then summing over all $i~~(i \geq 0)$, we get
	\begin{equation}\label{equt1}
	\sum_{i \geq 0}(n - i) c_{i}(G)x^{i}
	= \sum_{i \geq 0}\sum_{v \in V(G)}c_{i}(G - v)x^{i},\nonumber
	\end{equation}
	or equivalently, by interchanging the summation order, we have  
	\begin{equation}\label{equ2}
	n\sum_{i \geq 0}c_{i}(G)x^{i} - x \sum_{i \geq 1}ic_{i}(G)x^{i-1} 
	= \sum_{v \in V(G)}\left( \sum_{i \geq 0}c_{i}(G - v)x^{i} \right). 
	\end{equation}
	On the other hand, by the definition of a clique polynomial of a 
	graph, it's first derivative is equal to
	\begin{equation}\label{eq-derivat}
	\frac{d}{dx}C(G,x) = \sum_{i \geq 1}ic_{i}(G)x^{i-1} 
	. 
	\end{equation}
	Hence, considering relation (\ref{eq-derivat}) and the definition of a clique polynomial, we can rewrite
	equation (\ref{equt1}) as follows
	\begin{equation}
	n . C(G,x) - x \frac{d}{dx}C(G,x)   
	= \sum_{v \in V(G)}C(G - v, x),\nonumber
	\end{equation}
	or equivalently,
	\begin{equation}\label{equ4}
	\frac{d}{dx}C(G,x)   
	= \sum_{v \in V(G)}\frac{C(G , x) - C(G - v, x)}{x}.
	\end{equation}
	Thus, considering relations (\ref{equ4}) and (\ref{vert-rec1})
	we finally get the desired result.

\end{proof}

Next, we give the graph-theoretical interpretation of 
the second derivative of the clique polynomials 
which is a new combinatorial formula to the best of our 
knowledge. 

\begin{thm}
	For any simple graph 
	$G = (V, E)$, we have
	\begin{equation}
	\frac{1}{2!}\frac{d^{2}}{dx^{2}}C(G,x) = \sum_{e \in E(G)}C(G[N(e)], x ).  
	\end{equation}
\end{thm}
\begin{proof}
	We first multiplying both sides of formula (\ref{edge-rec2}) by 
	$x^{i}$ and then summing over all $i~~(i \geq 0)$, we get
	\begin{equation}
	\sum_{i \geq 0}\left( m - {i \choose 2} \right)  c_{i}(G)x^{i}
	= \sum_{i \geq 0}\sum_{e \in E(G)}c_{i}(G - e)x^{i},\nonumber
	\end{equation}
	or equivalently, by interchanging the summation order, we obtain  
	\begin{equation}\label{eequ3}
	m\sum_{i \geq 0}c_{i}(G)x^{i} - x^{2} \sum_{i \geq 2}{i \choose 2}c_{i}(G)x^{i-2} 
	= \sum_{e \in E(G)}\left( \sum_{i \geq 0}c_{i}(G - e)x^{i} \right). 
	\end{equation}
	Next, by the definition of a clique polynomial of a 
	graph, it's second derivative is equal to
	\begin{equation}\label{eequ4}
	\frac{1}{2!}\frac{d^2}{dx^{2}}C(G,x) = \sum_{i \geq 2}{i \choose 2}c_{i}(G)x^{i-2} 
	. 
	\end{equation}
	
	Therefore, considering relation (\ref{eequ4}) and the definition of a clique polynomial, we can rewrite
	equation (\ref{eequ3}) as follows
	\begin{equation}
	m . C(G,x) - x^{2}\left( \frac{1}{2!} \frac{d^2}{dx^{2}}C(G,x)   \right) 
	= \sum_{e \in E(G)}C(G - e, x)\nonumber. 
	\end{equation}
	or equivalently,
	\begin{equation}\label{eequ5}
	\frac{1}{2!}\frac{d^2}{dx^{2}}C(G,x)   
	= \sum_{e \in E(G)}\frac{C(G , x) - C(G - e, x)}{x^2}.
	\end{equation}
	Finally, the relations (\ref{eequ5}) and (\ref{edge-rec2})
	imply the desired result.   
\end{proof}

\section{Triangle-deck Identity and Third Derivative}

In this section, we are going to develop
a graph-theoretical interpretation 
of the third derivative of a 
clique polynomial. 
\\
The following key lemma from 
\cite{HoedeLi,Ibrahim} is essential in our 
future arguments. 
We recall that by $G-M$, where $M$ is the set of edges 
of $G$, we mean a graph obtained from $G$ by deleting 
only edges of $M$ from $G$. 

\begin{lem}
	Let $G=(V,E)$	
	be a simple graph and the set of edges $M$ induces 
	a clique in $G$, then we have 
	
	\begin{equation}
	C(G,x) = C(G-M,x) + \sum^{\vert M \vert}_{r=2}
	(-1)^{r}(r-1) x^{r} 
	\sum_{S \subseteq M~\vert S \vert={r \choose 2}}
	C(G[N_{G}(S)], x). 
	\end{equation}

\end{lem}

We are mainly interested the particular case 
where $M$ induces a triangle $\delta$. 

\begin{thm}
	For any simple graph $G=(V,E)$
	and a triangle 
	$\delta=\{e_{1},e_{2},e_{3}\} \in \Delta_{3}(G)$, we have 
	
	\begin{equation}\label{TriIdentrkey}
	C(G,x) = C(G-\delta,x) + I_{2}(G,x) x^{2} - 
	2 I_{3}(G,x) x^{3},
	\end{equation} 	
	in which 
	\begin{equation}
	I_{2}(G,x) = \sum^{3}_{i=1}C(G[N_{G}(e_{i})]),
	\hspace{0.4cm}
	I_{3}(G,x) = C(G[N_{G}(\delta)],x). 
	\end{equation}
	
\end{thm}

One of the interesting question that naturally 
arise is that of \emph{triangle-recurrence}
for clique polynomials. In other words, we 
are searching for the class of graphs 
for which the following triangle-recurrence 
relation holds. 

\begin{equation}
C(G,x) = C(G-\delta,x) + x^{3} C(G[N_{G}(\delta)],x).
\end{equation}

Note that the above condition is \emph{equivalent}
to the following clique-counting 
identity 

\begin{equation}
\sum^{3}_{i=1}C(G[N_{G}(e_{i})]) = 3x C(G[N_{G}(\delta)],x)
\end{equation}

In the case of having \emph{symmetric property} for the 
\emph{edge-neighborhoods}; that is 
$G[N_{G}(e_{i})]$ ($i=1,2,3$) are the same, 
the above condition reduced to the following 
simple one.
\begin{equation}
C(G[N_{G}(e)]) = x C(G[N_{G}(\delta)],x). 
\end{equation}

Next, we prove the following theorem. 

\begin{thm}
	Let $G=(V,E)$ be a connected $K_{5}$-free 
	graph. Then, we have
	
	\begin{equation}
	\frac{1}{3!}\frac{d^{3}}{dx^{3}} C(G,x) =
	\sum_{\delta \in \Delta_{3}(G)}
	C(G[N_{G}(\delta)],x). 
	\end{equation}

\end{thm}

\begin{proof}
	We first define 
	\begin{equation}
	Diff(G,x) =
	C(G,x) - C(G-\delta,x). 
	\end{equation}
	Now, considering the identity (\ref{TriIdentrkey}),
	we equivalently have
	
	\begin{equation}
	C(G[N_{G}(\delta)],x) = 
	\frac{I_{2}(G,x) x^{2} - Diff(G,x)}
	{2 x^{{3}}}. 
	\end{equation}
	Our ultimate goal is to prove that 
	\begin{equation}
	\sum_{\delta \in \Delta_{3}(G)} C(G[N_{G}(\delta)],x)
	= \frac{1}{3!}\frac{d^{3}}{dx^{3}} C(G,x). 
	\end{equation}
	
	Next, we note that since $G$ is a connected 
	$K_{5}$-free graph, then we conclude that 
	\begin{equation}
	C(G,x) = 
	1 + c_{1}(G) x + c_{2}(G) x^{2}
	+ c_{3}(G) x^{3} + c_{4}(G) x^{4}. 
	\end{equation}
	Based on the proof of $(1)$, one can also show that 
	\begin{equation}
	C(G-\delta,x) = 
	1 + c_{1}(G-\delta) x + c_{2}(G-\delta) x^{2}
	+ c_{3}(G-\delta) x^{3} + c_{4}(G-\delta) x^{4}, 
	\end{equation}
	in which 
	\begin{eqnarray}
	c_{1}(G-\delta) & = & c_{1}(G), \nonumber\\
	c_{2}(G-\delta) & = & c_{2}(G)-3, \nonumber\\
	c_{3}(G-\delta) & = & c_{3}(G)-
	\sum_{i=1}^{3}val_{G}(e_{i})+2, \nonumber\\
	c_{4}(G-\delta) & = & c_{4}(G)-
	\sum_{i=1}^{3}c_{2}(G[N_{G}(e_{i})])+
	2val_{G}(\Delta).
	\end{eqnarray} 	
	
	By considering above formulas, we obtain 
	\begin{eqnarray}
	Diff(G,x) &=& 3x^{2} + \big(c_{3}(G)-
	\sum_{i=1}^{3}val_{G}(e_{i})-2\big) x^{3}\nonumber\\
	& + & \big(
	\sum_{i=1}^{3}c_{2}(G[N_{G}(e_{i})])-
	2val_{G}(\Delta)
	\big) x^{4},
	\end{eqnarray} 
	and 
	\begin{eqnarray}
	N_{2}(G,x) & = & 
	\big(
	1 + val_{G}(e_{1}) x + c_{2}(G[N_{G}(e_{1})]) x^{2}
	\big) \nonumber\\
	& + & 
	\big(
	1 + val_{G}(e_{2}) x + c_{2}(G[N_{G}(e_{2})]) x^{2}
	\big) \nonumber\\
	& + & 
	\big(
	1 + val_{G}(e_{3}) x + c_{2}(G[N_{G}(e_{3})]) x^{2}
	\big). \nonumber
	\end{eqnarray}
	Finally, from $(2)$ and $(3)$, we get 
	\begin{equation}
	C(G[N_{G}(\Delta)],x) = 
	\frac{2 x^{3} + 2 val_{G}(\Delta) x^{4}}
	{2 x^{3}} = 1 + val_{G}(\Delta) x. 
	\end{equation}
	Thus, considering $(4)$ and clique handshaking lemma for
	$\omega(G)=4$, we get 
	
	\begin{eqnarray}
	\sum_{\delta \in \Delta_{3}(G)} 
	C(G[N_{G}(\Delta)],x) & = & 
	\vert \Delta_{3}(G) \vert + 4c_{4}(G) x \nonumber\\
	& = & 
	c_{3}(G) + 4c_{4}(G) \nonumber\\
	& = & \frac{1}{3!} \frac{d^{3}}{dx^{3}} C(G,x), 
	\end{eqnarray}
	which is the desired result.

\end{proof}

\section{Open Questions and Conjectures}

Here, we propose several interesting open questions 
and conjectures regarding subgraph-counting polynomials. 
The first open question is related to another similar 
clique-counting polynomial \cite{Li-Gutman}, as follows. 
\\
${\bf Open~question~1.}$
Let $c(G,x)$ be a clique-counting polynomial defined by 

\begin{equation}
c(G,x)= 1+ \sum_{k=1}^{n}c_{k}(G) x^{n-k}.
\end{equation}	
Can we find the similar formulas for
the first and second derivatives of $c(G,x)$? 	
\\
$\textbf{Conjecture~1.}$
Let $G=(V,E)$ be a simple graph. Then, we have 

\begin{equation}
\frac{d}{dx}c(G,x) = 
\sum_{v \in V(G)} c(G-v,x).
\end{equation}
Moreover, we also have 

\begin{equation}
\frac{1}{2!}\frac{d^{2}}{dx^{2}}c(G,x) = 
\sum_{e \in E(G)} c(G-e,x).
\end{equation}
\\
$\textbf{Open~question~2.}$

Do we have a \emph{triangle-version} of Lemma 
\ref{edgekey2}. In other words, is the following identity 
true in general?
\\	
Let $G=(V,E)$ be a graph on $m$ edges and $H$
be any graph with $k$ vertices without 
isolated vertices such that 
$
\vert V(H) \vert \leq \vert V(G) \vert 
$
,
$
\vert E(H) \vert \leq \vert E(G) \vert 
$
and 
$
\vert \Delta_{3}(H) \vert < \vert \Delta_{3}(G) \vert 
$
.
Then, we have 

\begin{equation}\label{trianglekey}
\Big(
t-{k \choose 3} 
\Big) c_{k}(G) =
\sum_{\delta  \in \Delta_{3}(G)}
c_{k}(G-\delta) \hspace{0.4cm}(k \geq 3),
\end{equation}

where $t$ denotes the number of \emph{triangles}
in $G$. 
\\
It is not worthy that when you delete a vertex 
or an edge it has no effects in other cliques.

\begin{defn}
	
	For a given graph $G=(V,E)$, it's \emph{triangle graph}
	denoted by $T(G)$ is the graph whose vertex set is the 
	set of triangles of $G$ and two vertices of $T(G)$
	are connected if their corresponding triangles share 
	an edge. 
	
\end{defn}

$\textbf{Conjecture~2.}$

For a given graph $G$ such that it's triangle graph 
$T(G)$ is an empty graph, then the 
\emph{triangle-version} of Lemma \ref{edgekey2} holds 
for $G$. 
\\
As a possible extension of the derivatives of the clique polynomial, we
believe that the following conjecture is also true. 
\\
$\textbf{Conjecture~3.}$

We have the following formula for the \emph{third derivative} of 
the clique polynomial:

\begin{equation}
\frac{1}{3!}\frac{d^{3}}{dx^{3}}C(G,x) = 
\sum_{\delta \in \Delta_{3}(G)} C(G-\delta,x),
\end{equation}

%bbbbbbbbbbbbbbbbbbbbbbbbbbbbbbbbbbbbbbbbbbbbbbbbbbbbbbbbbb
%===========================bibl

%===========================================

\end{document}